\documentclass[a4paper, 12pt]{amsart}

\usepackage[usenames,dvipsnames]{color}
\usepackage{amsthm,amsfonts,amssymb,amsmath,amsxtra}
\usepackage[all]{xy}
\SelectTips{cm}{}
\usepackage{xr-hyper}
\usepackage[colorlinks=
   citecolor=Black,
   linkcolor=Red,
   urlcolor=Blue]{hyperref}
\usepackage{verbatim}

\usepackage[margin=1.25in]{geometry}
\usepackage{mathrsfs}

\RequirePackage{xspace}
\RequirePackage{etoolbox}
\RequirePackage{varwidth}
\RequirePackage{enumitem}
\RequirePackage{tensor}
\RequirePackage{mathtools}
\RequirePackage{longtable}
\RequirePackage{multirow}

\def\ge{\geqslant}
\def\le{\leqslant}
\def\a{\alpha}

\def\d{\delta}

\def\e{\epsilon}

\def\s{\sigma}
\def\t{\tau}
\def\th{\theta}
\def\k{\kappa}
\def\l{\lambda}

\def\i{^{-1}}

\def\<{\langle}
\def\>{\rangle}

\newcommand{\sA}{\ensuremath{\mathscr{A}}\xspace}

\newcommand{\sR}{\ensuremath{\mathscr{R}}\xspace}

\newcommand{\fka}{\ensuremath{\mathfrak{a}}\xspace}

\newcommand{\fkR}{\ensuremath{\mathfrak{R}}\xspace}

\newcommand{\BC}{\ensuremath{\mathbb {C}}\xspace}

\newcommand{{\BG}}{\ensuremath{\mathbb {G}}\xspace}

\newcommand{{\BK}}{\ensuremath{\mathbb {K}}\xspace}

\newcommand{\BN}{\ensuremath{\mathbb {N}}\xspace}

\newcommand{\BR}{\ensuremath{\mathbb {R}}\xspace}
\newcommand{\BS}{\ensuremath{\mathbb {S}}\xspace}

\newcommand{\BZ}{\ensuremath{\mathbb {Z}}\xspace}

\newcommand{\CI}{\ensuremath{\mathcal {I}}\xspace}

\newcommand{\CK}{\ensuremath{\mathcal {K}}\xspace}

\newcommand{\RR}{\ensuremath{\mathrm {R}}\xspace}

\newcommand{\Ad}{{\mathrm{Ad}}}

\DeclareMathOperator{\Gal}{Gal}

\DeclareMathOperator{\Hom}{Hom}

\let\Im\relax
\DeclareMathOperator{\Im}{Im}

\def\Tr{{\mathrm{Tr}}}

\def\rig{\text{rig}}

\def\tW{\tilde W}

\newtheorem{theorem}{Theorem}
\newtheorem{proposition}[theorem]{Proposition}
\newtheorem{theoremA}{Theorem}

\newtheorem{lemma}[theorem]{Lemma}

\newtheorem{corollary}[theorem]{Corollary}

\theoremstyle{definition}

\newtheorem{example}[theorem]{Example}

\newtheorem{remark}[theorem]{Remark}

\numberwithin{equation}{section}
\numberwithin{theorem}{section}

\setitemize[0]{leftmargin=*,itemsep=\the\smallskipamount}
\setenumerate[0]{leftmargin=*,itemsep=\the\smallskipamount}

\renewcommand{\to}{%
   \ifbool{@display}{\longrightarrow}{\rightarrow}%
   }
\let\shortmapsto\mapsto
\renewcommand{\mapsto}{%
   \ifbool{@display}{\longmapsto}{\shortmapsto}%
   }
\newlength{\olen}
\newlength{\ulen}
\newlength{\xlen}
\newcommand{\xra}[2][]{%
   \ifbool{@display}%
      {\settowidth{\olen}{$\overset{#2}{\longrightarrow}$}%
       \settowidth{\ulen}{$\underset{#1}{\longrightarrow}$}%
       \settowidth{\xlen}{$\xrightarrow[#1]{#2}$}%
       \ifdimgreater{\olen}{\xlen}%
          {\underset{#1}{\overset{#2}{\longrightarrow}}}%
          {\ifdimgreater{\ulen}{\xlen}%
             {\underset{#1}{\overset{#2}{\longrightarrow}}}
             {\xrightarrow[#1]{#2}}}}%
      {\xrightarrow[#1]{#2}}
   }
\makeatother
\newcommand{\xyra}[2][]{%
   \settowidth{\xlen}{$\xrightarrow[#1]{#2}$}%
   \ifbool{@display}%
      {\settowidth{\olen}{$\overset{#2}{\longrightarrow}$}%
       \settowidth{\ulen}{$\underset{#1}{\longrightarrow}$}%
       \ifdimgreater{\olen}{\xlen}%
          {\mathrel{\xymatrix@M=.12ex@C=3.2ex{\ar[r]^-{#2}_-{#1} &}}}%
          {\ifdimgreater{\ulen}{\xlen}%
             {\mathrel{\xymatrix@M=.12ex@C=3.2ex{\ar[r]^-{#2}_-{#1} &}}}
             {\mathrel{\xymatrix@M=.12ex@C=\the\xlen{\ar[r]^-{#2}_-{#1} &}}}}}%
      {\mathrel{\xymatrix@M=.12ex@C=\the\xlen{\ar[r]^-{#2}_-{#1} &}}}%
   }
\makeatletter
\newcommand{\xla}[2][]{%
   \ifbool{@display}%
      {\settowidth{\olen}{$\overset{#2}{\longleftarrow}$}%
       \settowidth{\ulen}{$\underset{#1}{\longleftarrow}$}%
       \settowidth{\xlen}{$\xleftarrow[#1]{#2}$}%
       \ifdimgreater{\olen}{\xlen}%
          {\underset{#1}{\overset{#2}{\longleftarrow}}}%
          {\ifdimgreater{\ulen}{\xlen}%
             {\underset{#1}{\overset{#2}{\longleftarrow}}}
             {\xleftarrow[#1]{#2}}}}%
      {\xleftarrow[#1]{#2}}
   }
\newcommand{\isoarrow}{%
   \ifbool{@display}{\overset{\sim}{\longrightarrow}}{\xrightarrow\sim}%
   }
   
\begin{document}

\title[]{Cocenter of $p$-adic groups, II: induction map}
\author[X. He]{Xuhua He}
\address{Department of Mathematics, University of Maryland, College Park, MD 20742 and Institute for Advanced Study, Princeton, NJ 08540}
\email{xuhuahe@math.umd.edu}

\thanks{X. H. was partially supported by NSF DMS-1463852 and DMS-1128155 (from IAS)}


\date{\today}

\begin{abstract}
In this paper, we study some relation between the cocenter $\bar H(G)$ of the Hecke algebra $H(G)$ of a connected reductive group $G$ over an nonarchimedean local field and the cocenter $\bar H(M)$ of its Levi subgroups $M$. 

Given any Newton component of $\bar H(G)$, we construct the induction map $\bar i$ from the corresponding Newton component of $\bar H(M)$ to it. We show that this map is surjective. This leads to the Bernstein-Lusztig type presentation of the cocenter $\bar H(G)$, which generalizes the work \cite{HN2} on the affine Hecke algebras. We also show that the map $\bar i$ we constructed is adjoint to the Jacquet functor and in characteristic $0$, the map $\bar i$ is an isomorphism. 
\end{abstract}

\maketitle


\section*{Introduction}

\subsection{} Let $\BG$ be a connected reductive group over a nonarchimedean local field $F$ of arbitrary characteristic and $G=\BG(F)$. Let $R$ be an algebraically closed field of characteristic not equal to $p$, where $p$ is the characteristic of residue field of $F$. Let $H_R$ be the Hecke algebra of $G$ over $R$ and $\bar H_R=H_R/[H_R, H_R]$ be its cocenter. Let $\fkR(G)_R$ be the $R$-vector space with basis the isomorphism classes of irreducible smooth admissible representations of $G$ over $R$. Then we have the trace map $$\Tr_R: \bar H_R \to \fkR(G)_R^*.$$

On the representation side, we have the induction functor and the Jacquet functor $$i_{M, R}: \fkR(M)_R \to \fkR(G)_R, \qquad r_{M, R}: \fkR(G)_R \to \fkR(M)_R,$$ where $M$ is a Levi subgroup of $G$.

What happens on the cocenter side? 

The functor adjoint to the induction functor $i_M$ is the restriction map $\bar r_{M, R}: \bar H(G)_R \to \bar H(M)_R$. It can be expressed explicitly via the Van Dijk's formula. In this paper, we investigate the functor $\bar i_{M, R}: \bar H_R(M) \to \bar H_R(G)$, which is adjoint to the Jacquet functor $r_{M, R}: \fkR(G)_R \to \fkR(M)_R$. 

\subsection{} We first describe the properties we expect for the map $\bar i_{M, R}$ and then discuss the approach toward it. 

First, instead of working over various algebraically closed fields $R$, it is desirable to have the map $\bar i_M$ defined on the integral form $\bar H$ (the cocenter of the Hecke algebra of $\BZ[\frac{1}{p}]$-valued functions). Such map, if exists, provides not only a uniform approach to the map $\bar i_{M, R}$ for all $R$, but also some useful information on the mod-$l$ representations (see Theorem \ref{thmD} in the introduction and a future work \cite{hecke-3} for some results in this direction). 

Second, in \cite{hecke-1}, we introduced the Newton decomposition. Roughly speaking, $$G=\sqcup G(v) \quad \text{ and } \quad \bar H=\oplus \bar H(v),$$ where $v$ runs over the set of dominant rational coweights of $G$. Such description is expected to play an important role in the representation theory of $p$-adic groups. In order to relate the Newton decomposition with the representations, we would like to know that the Newton decomposition is compatible with the map $\bar i_M$. 

\subsection{}Now we discuss several approaches in the literature towards the understanding of the map $\bar i_M$. 

Over $\BC$, the spectral density Theorem of Kazhdan \cite{Kaz} asserts that the trace map $\Tr_\BC: \bar H_\BC \to \fkR(G)_\BC^*$ is injective. Hence the map $\bar i_{M, \BC}$ is uniquely determined by the adjunction formula $$\Tr^M_\BC(f, r_{M, \BC}(\pi))=\Tr^G_\BC(\bar i_{M, \BC}(f), \pi).$$ However, if $R$ is of positive characteristic, the trace map $\Tr_R$ may not be injective and thus the map $\bar i_{M, R}$ is not uniquely determined by the adjunction formula. 

In those cases, one may use the categorical description of the cocenter to give a definition of $\bar i_{M, R}$. Bernstein's second adjointness theorem implies that the map $\bar i_{M, R}$ defined in this way is adjoint to the Jacquet functor (see \cite[(1.8)]{Dat}). However, it is not clear that this map preserves the integral structure (see some discussion in \cite[\S 4.27]{Dat}). Also it is not clear if this description is compatible with the Newton decomposition. 

\subsection{} A different, but more explicit approach is given by Bushnell in \cite{Bu}. 

Note that the induction functor $i_{M, R}$ on the representations of $M$ depends not only on the Levi subgroup $M$, but also on the parabolic subgroup $P$ with Levi factor $M$. However, when passing to the Grothendieck group of the representations, the dependence of $P$ disappears. On the other hand, the Jacquet functor $r_{M, R}$, even if one passes to the Grothendieck groups of the representations, still depend on the choice of parabolic subgroup. 

Let $v$ be a rational coweight. Then $v$ determines a Levi subgroup $M=M_v$ and the parabolic subgroup $P_v=M N_v$. Let $\CK$ be a ``nice'' open compact subgroup of $G$ (e.g. the $n$-th congruent subgroup $\CI_n$ of an Iwahori subgroup) and $\CK_M=\CK \cap M$. Bushnell introduced the $P_v$-positive elements of $M$ and the subalgebra $H^v(M, \CK_M)$ of $H(M, \CK_M)$, consisting of compactly supported $\CK_M$-biinvariant functions supported in the $P_v$-positive elements. Then he proves that 

(a) The algebra $H(M, \CK_M)$ is isomorphic to the localization of $H^v(M, \CK_M)$ at a strongly positive element $f_z$. 

(b) The map $$j_{v, \CK}: H^v(M, \CK_M) \to H(G, \CK), \d_{\CK_M m \CK_M} \mapsto \d_{P_v}(m)^{-\frac{1}{2}} \frac{\mu_G(\CK)}{\mu_M(\CK_M)} \d_{\CK m \CK}$$ is an injective algebra homomorphism.

(c) The map $j_{v, \CK}$ is adjoint to the Jacquet functor $r_{M, \CK, R}: \fkR_\CK(G)_R \to \fkR_{\CK \cap M}(M)_R$ relative to $P_v$. Here $\fkR_\CK(G)_R \subset \fkR(G)_R$ consists of representations generated by their $\CK$-fixed vectors.

Moreover, Bushnell's map $j_{v, \CK}$ also preserves the integral structure of the Hecke algebra. 

\subsection{} It is tempting to apply Bushnell's result to the cocenter of Hecke algebras. However, there are several obstacles. 

If $\CK$ is the Iwahori or pro-$p$ Iwahori subgroup, then the map $j_{v, \CK}$ extends to an algebra homomorphism $H(M, \CK \cap M) \to H(G, \CK)$. In this case, the localization of Hecke algebra $H^v(M, \CK \cap M)$ is consistent with the Bernstein-Lusztig presentation (\cite{L} and \cite{V}). However, as pointed out in \cite{Bu}, these are essentially the only cases of this kind. Thus one may only use $j_{v, \CK}$ to deduce the induction map from part of the cocenter of $H(M)$ to the cocenter of $H(G)$.

The Newton strata of $M$ with integral dominant Newton points are positive, but the strata with rational (but not integral) Newton point may not be positive for any parabolic $P$. Those strata are not in the domain of the maps $j_{v, \CK}$. 

Also if one fixes $M$ and $P$, the maps $j_{v, \CK}$ are not compatible with the change of open compact subgroups $\CK$, even at the cocenter level (see \S\ref{non-com}). Thus the maps $j_{v, \CK}$ does not induce a well-defined map $\bar H^v(M) \to \bar H$. 

\subsection{} The idea behind Bushnell's map $j_{v, \CK}$ is to enlarge the open compact subset $\CK_M m \CK_M$ of $M$ to the open compact subset $\CK m \CK$ of $G$ by multiplying the open compact subgroup $\CK$. Inspired by it, we have the following construction. 

Let $v$ be a rational coweight and $P=M N_v$ be the associated parabolic subgroup. The elements in the Newton stratum $M(v)$ may not be $P_v$-positive, but a sufficiently large power of it is $P_v$-positive. One may enlarge an open compact subset inside $M(v)$ by multiplying a suitable open compact subgroup of $G$ to obtain an open compact subset of $G$. Unlike the situation in \cite{Bu}, the lack of $P_v$-positivity condition prevents us to give an explicit open compact subgroup of $G$ that works in our situation. We have to use sufficiently small open compact subgroup of $G$. Since $v$ is strictly positive with respect to $N_v$, we finally show that our construction is independent of the choice of such open compact subgroups. We have 

\begin{theoremA}\label{thmA}
Let $v$ be a rational coweight and $M=M_v$. Let $\bar v$ be the $G$-dominant coweight associated to $v$. Then 

(1) [Theorem \ref{levi}] The map $$\d_{m \CK_M} \mapsto \d_{P_v}(m)^{-\frac{1}{2}} \frac{\mu_M(\CK_M)}{\mu_G(\CK_M \CK)} \d_{m \CK_M \CK}+[H, H]$$ for sufficiently small open compact subgroup $\CK$ of $G$ gives a well-defined map $$\bar i_v: \bar H(M; v) \to \bar H.$$

(2) [Theorem \ref{BL}] The image of $\bar i_v$ equals $\bar H(G; \bar v)$. 

(3) [Theorem \ref{inj-i}] If moreover, $\text{char}(F)=0$, then the map $\bar i_v$ gives a bijection between $\bar H(M; v)$ and $\bar H(G; \bar v)$. 
\end{theoremA}

\begin{theoremA}[Theorem \ref{adj}] \label{thmB}
Let $v$ be a rational coweight and $M=M_v$. Then for any $f \in \bar H_R(M; v)$ and $\pi \in \fkR(G)_R$, we have the following adjunction formula $$\Tr^M_R(f, r_{v, R}(\pi))=\Tr^G_R(\bar i_v(f), \pi).$$ Here $r_{v, R}: \fkR(G)_R \to \fkR(M)_R$ is the Jacquet functor relative to $P_v$. 
\end{theoremA}

\subsection{} Now we discuss some applications. In \cite{hecke-1}, we introduced the rigid cocenter $\bar H^\rig=\oplus \bar H(v)$, where $v$ runs over rational central coweights. 

Now for any standard Levi subgroup $M$, we introduce the $+$-rigid part $\bar H(M)^{+, \rig}=\oplus \bar H(M; v)$, where $v$ runs over rational dominant coweights with $M=M_v$. We then have the well-defined map $$\bar i_M^+=\oplus_v \bar i_v: \bar H(M)^{+, \rig} \to \bar H.$$ As an application of Theorem \ref{thmA} and the Newton decomposition of $
\bar H$ (see \cite[Theorem 3.1]{hecke-1}), we have 

\begin{theoremA}\label{thmC}
We have the decomposition of the cocenter $\bar H$ into $+$-rigid parts: $$\bar H=\oplus_{M \text{ is a standard Levi subgroup}} \,  \bar i_M^+(\bar H(M)^{+, \rig}).$$
\end{theoremA}

For affine Hecke algebras, such decomposition is first obtained in \cite{HN2} via an elaborate analysis on the minimal length elements in the affine Weyl groups of $G$ and its Levi subgroups $M$. In loc.cit., such decomposition is called the Bernstein-Lusztig presentation of the cocenter of affine Hecke algebras, since the explicit expression of $\bar i_M^+$ there is given in terms of the Bernstein-Lusztig presentation. Although there is no Bernstein-Lusztig type presentation for $H$, we follow \cite{HN2} and still call the decomposition in Theorem \ref{thmC} the Bernstein-Lusztig presentation of the cocenter $\bar H$. It is also worth mentioning that the proof in this paper does not involve the elaborate analysis on the minimal length elements as in \cite{HN2},  but based on the compatibility between the change of different open compact subgroups $\CK$ of $G$. 

Theorem \ref{thmC} asserts that the rigid cocenters of Levi subgroups form the ``building blocks'' of the whole cocenter $\bar H$. We also show that that they are compatible with the trace map in the following way. 

\begin{theoremA}[Theorem \ref{newton-ker}] \label{thmD}
Let $R$ be an algebraically closed field of characteristic not equal to $p$. Then we have $$\ker \Tr_R=\oplus_{M \text{ is a standard Levi subgroup}} \,  \bar i_M^+(\ker \Tr^M_R \cap \bar H_R(M)^{+, \rig}).$$
\end{theoremA}  

If $R=\BC$, we have the spectral density theorem and the kernel of the trace map is zero. Theorem \ref{thmD} is trivial in this case. However, if $R$ is of positive characteristic, especially when the spectral density theorem fails, then Theorem \ref{thmD} would provide useful information toward the understanding of those representations. 

\subsection{} The outline of the proof is as follows. In \S\ref{2}, we introduce the notion of quasi-positive elements and we use some remarkable properties on the minimal length elements established in \cite{HN1} to show that any element in the Newton stratum $M(v)$ is quasi-positive. Then in \S\ref{3}, we use the quasi-positivity to show that the map in Theorem \ref{thmA} (1) is well-defined and factors through $\bar H(M; v)$. This proves part (1) of Theorem \ref{thmA}. 

As to part (2) of Theorem \ref{thmA}, we first prove in Proposition \ref{M-G-nu} that $M(v) \subset G(\bar v)$. Then by the admissibility of Newton strata (\cite[Theorem 3.2]{hecke-1}), any open compact subset $X$ of $M(v)$ enlarged by a sufficiently small open compact subgroup is still contained in $G(\bar v)$. This shows that the image of $\bar i_v$ is contained in $\bar H(G; \bar v)$. The key ingredients in the proof of surjectivity are 
\begin{itemize}
\item The notation of $P$-alcove elements introduced in \cite{GHKR}. 

\item The Iwahori-Matsumoto presentation of $\bar H(G; \bar v)$ (\cite[Theorem 4.1]{hecke-1}). 
\end{itemize}

By the quasi-positivity, for any $f \in H(M; v)$, $f^l \in H^v(M)$ for sufficiently large $l$. Theorem \ref{thmB} follows from the adjunction formula proved in \cite{Bu}, the comparison between $i_v(f)^l$ with $j_{v, *}(f^l)$ and a trick of Casselman \cite{Ca}. 

Finally, the injectivity in part (3) of Theorem \ref{thmA} follows from the adjunction formula (Theorem \ref{thmB}), the spectral density theorem and the freeness of the cocenter $\bar H$ (which is only known in the case of $\text{char}(F)=0$). 

\subsection{Acknowledgment} We thank Dan Ciubotaru and Sian Nie for many enjoyable discussions and valuable suggestions. We also thank Guy Henniart and Marie-France Vign\'eras for discussions on the freeness of cocenter and thank Maarten Solleveld for his useful comments on a preliminary version of the paper. 

\section{Preliminary}

\subsection{} Let $\BG$ be a connected reductive group over a nonarchimedean local field $F$ of arbitrary characteristic. Let $G=\BG(F)$. We fix a maximal $F$-split torus $A$ and an alcove $\fka_C$ in the corresponding apartment, and denote by $\CI$ the associated Iwahori subgroup. 

Let $Z=Z_G(A)$. We denote by $W_0=N_G A(F)/Z(F)$ the {\it relative Weyl group} and $\tW=N_G A(F)/Z_0$ the {\it Iwahori-Weyl group}, where $Z_0$ is the unique parahoric subgroup of $Z(F)$. 

We fix a special vertex of $\fka_C$ and identify $\tW$ as $$\tW \cong X_*(Z)_{\Gal(\bar F/F)} \rtimes W_0=\{t^\l w; \l \in X_*(Z)_{\Gal(\bar F/F)}, w \in W_0\}.$$

We have a semidirect product
$$\tW=W_a \rtimes \Omega,$$ where $W_a$ is the affine Weyl group associated to $\tW$ and $\Omega$ is the stabilizer of the alcove $\fka_C$ in $\tW$. Let $\tilde \BS$ be the set of affine simple reflections of $W_a$ determined by the fundamental alcove $\fka_C$. The groups $W_a$ and $\tW$ are equipped with a Bruhat order $\le$ and a length function $\ell$. The subgroup $\Omega$ of $\tW$ is the subgroup consisting of length-zero elements. 

\subsection{} For any $K \subset \tilde \BS$, let $W_K$ be the subgroup of $\tW$ generated by $s \in K$. Let ${}^K \tW$ be the set of elements $w \in \tW$ of minimal length in the cosets $W_K w$. 

Let $\Phi=\Phi(G, A)$ be the set of roots of $G$ relative to $A$ and $\Phi^+$ be the set of positive roots so that $\fka_C$ is contained in the antidominant chamber of $V$ determined by $\Phi^+$. Let $\sR=\{\a\}$ be the set of affine roots on $\sA$. We choose a normalization of the valuation on $F$ so that if $\a \in \sR$, then so is $\a \pm 1$ (see \cite[\S 5.2.23]{BT}). For any $n \in \BN$, let $\CI_n$ be the $n$-th Moy-Prasad subgroup associated to the barycenter of $\fka_C$ \cite{MP}. This is the subgroup of $G$ generated by the $n$-th congruence subgroup of $Z(F)$ and the affine root subgroup $X_{\a+n}$ for $\a \in \sR_+$. 

For any $n \in \BN$ and a subgroup $G'$ of $G$, we set $G'_n=G' \cap \CI_n$. We write $\CI_{G'}$ for $G' \cap \CI$. 

\subsection{} Let $\mu_G$ be the Haar measure on $G$ such that the pro-p Iwahori subgroup $\CI'$ has volume $1$. As in \cite[Section 1]{hecke-1}, we denote by $H=H(G)$ the Hecke algebra of locally constant, compactly supported $\BZ[\frac{1}{p}]$-valued functions on $G$. We have $$H=\varinjlim\limits_{\CK} H(G, \CK),$$ where $\CK$ runs over open compact subgroups of $G$ and $H(G, \CK)$ is the space of compactly supported, $\CK \times \CK$-invariant $\BZ[\frac{1}{p}]$-valued functions on $G$, i.e., $H(G, \CK)=\oplus_{g \in \CK \backslash G/\CK} \BZ[\frac{1}{p}] \d_{\CK g \CK}$, where $\d_{\CK g \CK}$ is the characteristic function on $\CK g \CK$. 

We define the action of $G$ on $H$ by ${}^x f(g)=f(x \i g x)$ for $f \in H$, $x, g \in G$. By \cite[Proposition 1.1]{hecke-1}, the commutator $[H, H]$ of $H$ is the $\BZ[\frac{1}{p}]$-submodule of $H$ spanned by $f-{}^x f$ for $f \in H$ and $x \in G$. Let $\bar H=H/[H, H]$ be the cocenter of $H$. 

\subsection{} Now we recall the Newton decomposition introduced in \cite{hecke-1}.

Set $V=X_*(Z)_{\Gal(\bar F/F)} \otimes \BR$ and $V_+$ be the set of dominant elements in $V$. For any $w \in \tW$, there exists a positive integer $l$ such that $w^l=t^\l$ for some $\l \in X_*(Z)_{\Gal(\bar F/F)}$. We set $\nu_w=\l/l \in V$ and $\bar \nu_w$ to be the unique dominant in the $W_0$-orbit of $\nu_w$. The element $\nu_w$ and $\bar \nu_w$ are independent of the choice of $l$. 

Let $\aleph=\Omega \times V_+$. We have a map (see \cite[\S 2.1]{hecke-1}) $$\pi=(\k, \bar \nu): \tW \to \aleph, \qquad w \mapsto (w W_a, \bar \nu_w).$$  

We denote by $\tW_{\min}$ be the subset of $\tW$ consisting of elements of minimal length in their conjugacy classes. For any $\nu \in \aleph$, we set $$X_\nu=\cup_{w \in \tW_{\min}; \pi(w)=\nu} \CI \dot w \CI \quad \text{ and } \quad G(\nu)=G \cdot_\th X_\nu.$$ Here $\cdot$ means the conjugation action of $G$. Let $H(\nu)$ be the submodule of $H$ consisting of functions supported in $G(\nu)$ and let $\bar H(\nu)$ be the image of $H(\nu)$ in the cocenter $\bar H$. The Newton decomposition of $\bar H$ is established in \cite[Theorem 3.1 (2)]{hecke-1}.

\begin{theorem}\label{newton-h}
We have that $$\bar H=\bigoplus_{\nu \in \aleph} \bar H(\nu).$$
\end{theorem}

In this paper, we are mainly interested in the $V$-factor of $\aleph$. For any $v \in V_+$, we also set $G(v)=\sqcup_{\nu=(\t, v) \text{ for some }\t \in \Omega} G(\nu)$, $H(v)=\oplus_{\nu=(\t, v) \text{ for some }\t \in \Omega} H(\nu)$ and $\bar H(v)=\oplus_{\nu=(\t, v) \text{ for some }\t \in \Omega} \bar H(\nu)$. 

\subsection{} Let $M$ be a semistandard Levi subgroup of $G$, i.e., a Levi subgroup of some parabolic subgroup of $G$ that contains $Z$. Let $\CI_M=\CI \cap M$ be the Iwahori subgroup of $M$ and $\tW(M)$ be the Iwahori-Weyl group of $M$. We denote by $\tilde \BS(M)$ the set of affine simple reflections of $\tW(M)$ determined by the Iwahori subgroup $\CI_M$. 

We may regard $\tW(M)$ as a subgroup of $\tW$ in a natural way. However, the length function $\ell_M$ on $\tW(M)$ does not equal to the restriction ot $\tW$ of the length function $\ell$ on $\tW$. 

Let $\Omega_M$ be the subgroup of $\tW(M)$ consisting of length-zero elements with respect to the length function $\ell_M$. We have $\Omega_M \cong \tW(M)/W_a(M)$, where $W_a(M)$ is the affine Weyl group of the subgroup of $\tW(M)$. We have $W_a(M) \subset W_a$ and thus a natural map $\Omega_M \cong \tW(M)/W_a(M) \to \tW/W_a \cong \Omega.$ Let $V_+^M$ be the set of $M$-dominant elements in $V$. We set $\aleph_M=\Omega_M \times V_+^M$ and we have a map $\pi_M=(\k_M, \bar \nu_M): \tW(M) \to \aleph_M.$

We also have a natural map $\aleph_M \to \aleph$ sending $(\t, v)$ to $(\t', \bar v)$, where $\t'$ is the image of $\t$ in $\Omega$ and $\bar v$ is the unique ($G$-)dominant element in the $W_0$-orbit of $v$. 

Let $\mu_M$ be the Haar measure on $M$ such that the pro-p Iwahori subgroup of $M$ has volume $1$. Let $H(M)$ be the Hecke algebra of $M$ and $\bar H(M)$ be its cocenter. For any $\nu_M \in \aleph_M$, we denote by $\bar H(M; \nu_M)$ the corresponding Newton component of $\bar H(M)$. By Theorem \ref{newton-h}, we have $$\bar H(M)=\oplus_{\nu_M \in \aleph_M} \bar H(M; \nu_M).$$

\section{Quasi-positive elements}\label{2}

\subsection{} The semistandard Levi may be described as the centralizer of elements in $V$. For any $v \in V$, we set $\Phi_{v, 0}=\{a \in \Phi; \<a, v\>=0\}$ and $\Phi_{v, +}=\{a \in \Phi; \<a, v\>>0\}$. Let $M_v \subset G$ be the Levi subgroup generated by $Z$ and $U_a(F)$ for $a \in \Phi_{v, 0}$ and $N_v \subset G$ be the unipotent subgroup generated by $U_a(F)$ for $a \in \Phi_{v, +}$. Set $P_v=M_v N_v$. Then $P_v$ is a semistandard parabolic subgroup and $M_v$ is a Levi subgroup of $P_v$. We denote by $P^-_v=M_v N^-_v$ the opposite parabolic. Let $\mu_{N_v}$, $\mu_{N_v^-}$ be the Haar measures on $N_v$ and $N_v^-$ respectively such that $\mu_G(n m n^-)=\mu_{N_v}(n) \mu_{M_v}(m) \mu_{N_v^-}(n^-)$ for $n \in N_v, m \in M_v, n^- \in N_v^-$. For $m \in M_v$, set $\d_v(m)=\frac{\mu_{N_v} (m N_{v, 0} m \i) }{\mu_{N_v}(N_{v, 0})}$. For $\nu=(\t, v) \in \aleph$, we may also write $M_\nu$ for $M_v$, $N(\nu)$ for $N_v$ and $N^-(\nu)$ for $N^-_v$. 

If $v$ is dominant, then $P_v$ is a standard parabolic subgroup of $G$ and $M_v$ is a standard Levi subgroup of $G$.

\subsection{} Let $v \in V$. Following \cite[Definition 6.5 \& Definition 6.14]{BK}, we call an element $m \in M_v$ a {\it $(P_v, \CI_n)$-positive} element if $$m N_{v, n} m \i \subset N_{v, n}, \text{ and } m \i N^-_{v, n} m \subset N^-_{v, n}.$$

We call an element $z$ in the center of $M_v$ a {\it strongly $P_v$-positive} element if the sequences $z^n N_{v, 0} z^{-n}, z^{-n} N^-_{v, 0} z^n$ both tend monotonically to $1$ as $n \to \infty$. 

Following \cite[\S 3.1]{Bu}, let $H^v(M_v, M_{v, n})$ be the subalgebra of $H(M_v, M_{v, n})$ of functions with support consisting of $(P_v, \CI_n)$-positive elements. The following result is proved in \cite[Proposition 5]{Bu}.

\begin{proposition}
The map $\d_{M_{v, n} m M_{v, n}} \mapsto \d_v(m)^{-\frac{1}{2}} \frac{\mu_{M_v} (M_{v, n})}{\mu_G(\CI_n)} \d_{\CI_n m \CI_n}$ defines an injective algebra homomorphism $$j_{v, n}: H^v(M_v, M_{v, n}) \hookrightarrow H(G, \CI_n).$$ 
\end{proposition}

The formula we have here differs from \cite{Bu} by the factor $\d_v(m)^{-\frac{1}{2}}$, since in \cite{Bu} the map is adjoint to the (unnormalized) Jacquet functor while we consider the (normalized) Jacquet functor. 

By \cite[\S 3.2]{Bu}, $H(M_v, M_{v, n})=S \i H^v(M_v, M_{v, n})$ is the localization of $H^v(M_v, M_{v, n})$, where $S=\<\d_{M_{v, n} z M_{v, n}}\>$ is the the multiplicative closed set of the function $\d_{M_{v, n} z M_{v, n}}$ with a strongly $P_v$-positive element $z$. It is pointed out in \cite[Remark 5]{Bu} that the map $j_{v, n}$ does not extend to an algebra homomorphism $H(M_v, M_{v, n}) \to H(G, \CI_n)$ for $n>0$. 

\subsection{} Let $v \in V$ be a rational coweight and $M=M_v$. For any $l \in \BN$ with $l v \in X_*(Z)$, the element $t^{l v}$ is strongly $P_v$-positive. However, in general, the element in $M(v)$ may not be $(P_v, *)$-positive. Therefore, one can not deduce a map from $\bar H(M; v)$ to $\bar H$ via the map $j_{v, n}$. 

\begin{example}
Let $G$ be split $GL_5$ and $M=GL_3 \times GL_2$. Let $v=(\frac{2}{3}, \frac{2}{3}, \frac{2}{3}, \frac{1}{2}, \frac{1}{2})$. Then $M=M_v$. The element $w=t^{(1, 1, 0, 1, 0)} (1 3 2) (4 5)$ of $\tW$ has Newton point $v$. However, $w(e_4-e_3)=e_5-e_2-1$ is a negative affine root. Therefore the element $\dot w$ is not $(P_v, *)$-positive. 
\end{example}

\subsection{}\label{q-p} To overcome the difficulty, we introduce the quasi-positive elements. 

An element $m \in M_v$ is called {\it $P_v$ quasi-positive} if there exists $l \in \BN$ such that \[\tag{a} m^l N_{v, n} m^{-l} \subset N_{v, n+1}, \text{ and } m^{-l} N^-_{v, n} m^l \subset N^-_{v, n+1} \text{ for any } n \in \BN.\]

For any $n \in \BN$, $w \in \tW$ and $g \in \CI \dot w \CI$,  we have $g \CI_{n+\ell(w)} g \i \subset \CI_n$. So

(b) Let $w \in \tW(M)$ and $m \in \CI_M \dot w \CI_M$. If $m$ satisfies (a), then we have \[m^n N_{v, n'+(l-1) \ell(w)} m^{-n} \subset N_{v, n'}, \text{ and } m^{-n} N^-_{v, n'+(l-1) \ell(w)} m^n \subset N^-_{v, n'} \text{ for any $n, n' \in \BN$}.\]

We first discuss some properties on the quasi-positive elements. 

\begin{proposition}\label{Levi-1}
Let $v \in V$ and $M=M_v$. Let $w \in \tW(M)$ and $m \in \CI_M \dot w \CI_M$. Suppose that $m$ satisfying the inclusion relation \S\ref{q-p} (a). 

(1) For any $n \in \BN$, any element in $m \CI_{n+(l-1) \ell(w)}$ is conjugate by an element in $\CI_n$ to an element in $m M_{n+(l-1) \ell(w)}$.  

(2) For any $n, n' \in \BN$ and $g \in \CI_{n+(l-1) \ell(w)}$, we have $$\d_{m g M_{n+(l-1) \ell(w)} \CI_{n+(l-1) \ell(w)+n'}} \equiv \d_{m M_{n+(l-1) \ell(w)} \CI_{n+(l-1) \ell(w)+n'}} \mod [H, H].$$
\end{proposition}

\begin{proof}
(1) We first show that

(a) For any $i \in \BN$, any element in $m M_{n+(l-1) \ell(w)} \CI_{{n+(l-1) \ell(w)}+i}$ is conjugate by $\CI_{n+i}$ to an element in $m M_{n+(l-1) \ell(w)} \CI_{{n+(l-1) \ell(w)}+i+1}$. 

Note that any element in $m M_{n+(l-1) \ell(w)} \CI_{{n+(l-1) \ell(w)}+i}$ is conjugate by $\CI_{{n+(l-1) \ell(w)}+i}$ to an element of the form $u' g u$ with $u' \in N^-_{v, {n+(l-1) \ell(w)}+i}$, $g \in m M_{n+(l-1) \ell(w)}$ and $u \in N_{v, {n+(l-1) \ell(w)}+i}$. By \S\ref{q-p} (b), $g u g \i \in N_{v, {n+i}}$. We have $(u', g u g \i) \in (\CI_{{n+(l-1) \ell(w)}+i}, \CI_{n+i}) \subset \CI_{{n+(l-1) \ell(w)}+i+1}$. Now we have $$u' g u=u'(g u g \i) g \in (g u g \i) u' \CI_{{n+(l-1) \ell(w)}+i+1} g.$$ So $u' g u$ is conjugate by $\CI_{n+i}$ to an element in \begin{align*} u' \CI_{{n+(l-1) \ell(w)}+i+1} g (g u g \i) &=u' \CI_{{n+(l-1) \ell(w)}+i+1} (g^{2} u (g^{2}) \i) g \\ &=u' (g^{2} u (g^{2}) \i) \CI_{{n+(l-1) \ell(w)}+i+1} g \\ &=(g^{2} u (g^{2}) \i) u' \CI_{{n+(l-1) \ell(w)}+i+1} g.\end{align*} By the same procedure, for any $l \in \BN$, $u' g u$ is conjugate by $\CI_{n+i}$ to an element in $(g^{l} u (g^{l}) \i) u' \CI_{{n+(l-1) \ell(w)}+i+1} g$. By \S\ref{q-p} (a), $g^{l} u g^{-l}\in \CI_{{n+(l-1) \ell(w)}+i+1}$. Hence $u' g u$ is conjugate by $\CI_{n+i}$ to an element in $u'\CI_{{n+(l-1) \ell(w)}+i+1} g=u' g \CI_{{n+(l-1) \ell(w)}+i+1}$. By the same argument, any element in $u' g \CI_{{n+(l-1) \ell(w)}+i+1}$ is conjugate by $\CI_{n+i}$ to an element in $g \CI_{{n+(l-1) \ell(w)}+i+1}$. 

(a) is proved. 

Let $g_0 \in m M_n \CI_{n+(l-1) \ell(w)}$. By (a), we may construct inductively an element $z_i \in \CI_{n+i}$ for $i  \in \BN$ such that $g_{i+1}:=z_i \i g_i z_i$ is contained in $m M_{n+(l-1) \ell(w)} \CI_{n+(l-1) \ell(w)+i}$. The convergent product $z:=z_1 z_2 \cdots$ is a well-defined element in $\CI_{n}$ and $z \i g z \in m M_{n+(l-1) \ell(w)}$. 

(2) By part (1),  there exists $h \in \CI_{n+n'}$ such that $h m g h \i \in m M_{n+(l-1) \ell(w)}$. We have $(\CI_{n+n'}, M_{n+(l-1) \ell(w)}) \subset \CI_{n+n'+(l-1) \ell(w)}$. Therefore $M_{n+(l-1) \ell(w)} \CI_{n+(l-1) \ell(w)+n'}$ is a subgroup of $\CI$ and is stable under the conjugation action of $\CI_{n+n'}$. Thus $h m g M_{n+(l-1) \ell(w)} \CI_{n+(l-1) \ell(w)+n'} h \i=m M_{n+(l-1) \ell(w)} \CI_{n+(l-1) \ell(w)+n'}$. The statement is proved. 
\end{proof}

\subsection{}\label{non-com} We say that $m \in M$ is {\it $P_v$ strictly positive} if for any $n \in \BN$, we have $$m N_{v, n} m \i \subset N_{v, n+1}, \text{ and } m \i N^-_{v, n} m \subset N^-_{v, n+1}.$$ We denote by $H^{v^\sharp}(M)$ the subalgebra of $H(M)$ consisting of functions with support consisting of $P_v$ strictly positive elements. Note that the limit of the support of $j_{v, n}(\d_{Z_0})$ for $v$ dominant regular, as $n$ goes to infinite, is just $Z_0$ itself, but the support of $j_{v, n}(\d_{Z_0})$ for each $n$ contains of nonsplit regular semisimple elements. Thus the maps $\{j_{v, n}\}$ are not compatible with the natural maps $\bar H^v(M, M_n) \to \bar H^v(M, M_{n+1})$. 

However, we have the following compatibility result for $P_v$ strictly positive part. 

\begin{corollary}\label{str+}
Let $n \in \BN$. Then the following diagram commutes
\[\xymatrix{ \bar H^{v^\sharp}(M, M_n) \ar[r]^-{j_{v, n}} \ar@{^{(}->}[d] & \bar H(G, \CI_n) \ar@{^{(}->}[d] \\ \bar H^{v^\sharp}(M, M_{n+1}) \ar[r]^-{j_{v, n+1}} & \bar H(G, \CI_{n+1}).}\] 
\end{corollary}

\begin{proof}
Let $m \in M$ be $P_v$ strictly positive. Then $\d_{M_n m M_n} \in H^{v^\sharp}(M, M_n) \subset H^{v^\sharp}(M, M_{n+1})$. By definition, $$j_{v, n+1}(\d_{M_n m M_n})=\d_v(m)^{-\frac{1}{2}} \frac{\mu_{M} (M_{n+1})}{\mu_G(\CI_{n+1})} \d_{\CI_{n+1} M_n m M_n \CI_{n+1}}.$$

Note that $\CI_{n+1} M_n=M_n \CI_{n+1}$ is a subgroup of $\CI$. We have $$\CI_n m \CI_n=\sqcup_{(i_1, i_2, i'_1, i'_2)} i_1 i'_1 \CI_{n+1} M_n m M_n \CI_{n+1} i'_2 i_2,$$ where $\{(i_1, i_2, i'_1, i'_2)\} \subset N_n \times N_n \times N^-_n \times N^-_n$ is a finite subset. By Proposition \ref{Levi-1} (2), for $i_1, i_2 \in N_n$ and $i'_1, i'_2 \in N^-_n$, we have \[ \d_{i_1 i'_1 \CI_{n+1} M_n m M_n \CI_{n+1} i'_2 i_2} \equiv \d_{\CI_{n+1} M_n m M_n \CI_{n+1}} \mod [H, H].\] Thus $$j_{v, n}(\d_{M_n m M_n}) \equiv \d_v(m)^{-\frac{1}{2}} \frac{\mu_{M} (M_{n})}{\mu_G(\CI_{n})} \frac{\mu_G(\CI_n m \CI_n)}{\mu_G(\CI_{n+1} M_n m M_n \CI_{n+1})} \d_{\CI_{n+1} M_n m M_n \CI_{n+1}} \mod [H, H].
$$

It remains to show that $\frac{\mu_{M} (M_{n+1})}{\mu_G(\CI_{n+1})}=\frac{\mu_{M} (M_{n})}{\mu_G(\CI_{n})} \frac{\mu_G(\CI_n m \CI_n)}{\mu_G(\CI_{n+1} M_n m M_n \CI_{n+1})} $. 

Suppose that $m \in \CI_M \dot w \CI_M$ for some $w \in \tW(M)$. By \cite[Lemma 4.6]{hecke-1}, \begin{gather*} \frac{\mu_G(\CI_n m \CI_n)}{\mu_G(\CI_n)}=\frac{\mu_G(\CI_{n+1} m \CI_{n+1})}{\mu_G(\CI_{n+1})}=q^{\ell(w)}, \frac{\mu_M(M_n m M_n)}{\mu_M(M_n)}=\frac{\mu_M(M_{n+1} m M_{n+1})}{\mu_M(M_{n+1})}=q^{\ell_M(w)}. \end{gather*} Now we have 
\begin{align*} \frac{\mu_{M} (M_{n})}{\mu_G(\CI_{n})} \frac{\mu_G(\CI_n m \CI_n)}{\mu_G(\CI_{n+1} M_n m M_n \CI_{n+1})} &=\frac{\mu_{M} (M_{n})}{\mu_G(\CI_{n})} \frac{\mu_G(\CI_n m \CI_n)}{\mu_G(\CI_{n+1} m \CI_{n+1})} \frac{\mu_G(\CI_{n+1} m \CI_{n+1})}{\mu_G(\CI_{n+1} M_n m M_n \CI_{n+1})} \\ &=\frac{\mu_{M} (M_{n})}{\mu_G(\CI_{n})} \frac{\mu_G(\CI_n m \CI_n)}{\mu_G(\CI_{n+1} m \CI_{n+1})} \frac{\mu_M(M_{n+1} m M_{n+1})}{\mu_M(M_n m M_n)} \\ &=\frac{\mu_{M} (M_{n})}{\mu_G(\CI_{n})}  \frac{\mu_G (\CI_{n})}{\mu_G(\CI_{n+1})} \frac{\mu_{M} (M_{n+1})}{\mu_M(M_{n})}=\frac{\mu_{M} (M_{n+1})}{\mu_G(\CI_{n+1})} .
\end{align*}
The statement is proved.
\end{proof}

Finally we show that the elements in $M_\nu(\nu)$ are $P_\nu$ quasi-positive. 

\begin{proposition}\label{+1}
Let $v \in V$ be a rational coweight and $M=M_v$. Let $w \in \tW(M)$.  Then there exists a positive integer $i_{\nu,w}$ such that for any $m \in \CI_M \dot w \CI_M \cap M(v)$ and $n \ge i_{v, w}$, we have $$m^{i_{v, w}} N_{v, n} (m^{i_{v, w}}) \i \subset N_{v, n+1}, \quad (m^{i_{v, w}}) \i N^-_{v, n} m^{i_{v, w}} \subset N^-_{v, n+1}.$$ 
\end{proposition}

\subsection{}\label{basic-min} The proof relies on some remarkable properties of the Iwahori-Weyl group, which we recall here. 

For $w, w' \in \tW$ and $s \in \tilde \BS$, we write $w \xrightarrow{s} w'$ if $w'=s w s$ and $\ell(w') \le \ell(w)$.  We write $w \to w'$ if there is a sequence $w=w_0, w_1, \cdots, w_n=w'$ of elements in $\tW$ such that for any $1 \le k \le n$, $w_{k-1} \xrightarrow{s_k} w_k$ for some $s_k \in \tilde \BS$. We write $w \approx w'$ if $w \to w'$ and $w' \to w$. It is easy to see that if $w \to w'$ and $\ell(w)=\ell(w')$, then $w \approx w'$. We have that

(a) If $w \xrightarrow{s} w'$ and $\ell(w)=\ell(w')$, then for any $g \in \CI \dot w \CI$, there exists $g' \in \CI \dot s \CI$ such that $g' g (g') \i \in \CI \dot w' \CI$. 

(b) If $w \xrightarrow{s} w'$ and $\ell(w')<\ell(w)$, then for any $g \in \CI \dot w \CI$, there exists $g' \in \CI \dot s \CI$ such that $g' g (g') \i \in \CI \dot w' \CI \sqcup \CI \dot s \dot w \CI$.

An element $w \in \tW$ is called {\it straight} if $\ell(w^n)=n \ell(w)$ for any $n \in \BN$. A triple $(x, K, u)$ is called a {\it standard triple} if $x \in \tW$ is straight, $K \subset \tilde \BS$ with $W_K$ finite, $x \in {}^K \tW$ and $\Ad(x)(K)=K$, and $u \in W_K$. By definition, 

(c) For any $n \in \BN$ and $g_1, \cdots, g_n \in \CI \dot u \dot x \CI$, we have $g_1 g_2 \cdots g_n \in (\CI W_K \CI) (\CI \dot x^n \CI)$. 

It is proved in \cite[Theorem A \& Proposition 2.7]{HN1} that 
\begin{theorem}\label{min}
For any $w \in \tW$, there exists a standard triple $(x, K, u)$ such that $u x \in \tW_{\min}$ and $w \to u x$. In this case, $\pi(w)=\pi(x)$. 
\end{theorem}

Following \cite[\S 4.3]{GHN}, we write $w \overset s \rightharpoonup w'$ if either $w \overset s \to w'$ or $w'=s w$ and $\ell(w)>\ell(s w s)$, and we write $w \rightharpoonup w'$ if there exists a sequence $w=w_0, w_1, \cdots, w_n=w'$ of elements in $\tW$ such that for any $1 \le k \le n$, $w_{k-1} \overset{s_k} \rightharpoonup w_k$ for some $s_k \in \tilde \BS$. It is easy to see that if $w \in \tW_{\min}$ and $w \rightharpoonup w'$, then $w \approx w'$. 

We show that 

\begin{lemma}\label{seq-n}
Let $w \in \tW$ and $g \in \CI \dot w \CI$. Then there exists a standard triple $(x, K, u)$, a sequence $w=w_0, w_1, \cdots, w_n=u x$ of distinct elements in $\tW$ and a sequence $g=g_0, g_1, \cdots, g_n$ of elements in $G$ such that \begin{enumerate} 
\item $u x \in \tW_{\min}$;
\item for any $0 \le k \le n$, $g_k \in \CI \dot w_k \CI$;
\item for any $1 \le k \le n$, there exists $s_k \in \tilde \BS$ and $h_k \in \CI \dot s_k \CI$ such that $w_{k-1} \overset{s_k} \rightharpoonup w_k$ and $g_k=h_k g_{k-1} h_k \i$. 
\end{enumerate}
\end{lemma}

\begin{remark}\label{seq-n'}
By definition, if $w \rightharpoonup w'$, then $w' \in w W_a$ and $\ell(w') \le \ell(w)$. In particular, the length of the sequence is at most $\sharp\{x \in W_a; \ell(x) \le \ell(w)\}$. 
\end{remark}

\begin{proof}
We argue by induction on $\ell(w)$. 

If $w \in \tW_{\min}$, by Theorem \ref{min}, there exists a standard triple $(x, K, u)$ with $ux \in \tW_{\min}$ and a sequence $w=w_0, w_1, \cdots, w_n=u x$ of distinct elements in $\tW$ such that for any $1 \le k \le n$, $w_{k-1} \xrightarrow{s_k} w_k$ for some $s_k \in \tilde \BS$. Since $w \in \tW_{\min}$, we have $\ell(w_k)=\ell(w)$ for all $k$. Now the statement follows from \S\ref{basic-min} (a). 

If $w \notin \tW_{\min}$, then by Theorem \ref{min}, there exists a sequence $w=w_0, w_1, \cdots, w_n$ of distinct elements in $\tW$ such that $\ell(w)=\ell(w_n)$, for any $1 \le k \le n$, $w_{k-1} \xrightarrow{s_k} w_k$ for some $s_k \in \tilde \BS$ and there exists $s \in \tilde \BS$ with $s w_n s<w_n$. Then we have $\ell(w_k)=\ell(w)$ for all $k$. By \S\ref{basic-min} (a), for any $1 \le k \le n$, there exists $h_k \in \CI \dot s_k \CI$ such that $g_k=h_k g_{k-1} h_k \i$. By \S\ref{basic-min} (b), there exists $h_{n+1} \in \CI \dot s \CI$ such that $h_{n+1} g_n h_{n+1} \i \in \CI \dot w_{n+1} \CI$ with $w_{n+1} \in \{s w_n, s w_n s\}$. Now the statement follows from inductive hypothesis on $w_{n+1}$. 
\end{proof}

\subsection{Proof of Proposition \ref{+1}}
Let $N_0=\sharp\{w' \in W_a(M); \ell_M(w') \le \ell_M(w)\}$. By Lemma \ref{seq-n} and remark \ref{seq-n'}, there exists a standard triple $(x, K, u)$ of $\tW(M)$ and an element $h \in \cup_{z \in W_a(M); \ell(z) \le N_0} \CI_M \dot z \CI_M$ such that $u x \in \tW(M)_{\min}$, $w \rightharpoonup ux$ and $h m h \i \in \CI_M \dot u \dot x \CI_M$. 

Let $\mathfrak i$ be a positive integer with $\mathfrak i v \in X_*(Z)$.  Then $x^{\mathfrak i}=t^{\mathfrak i v} \in \tW$ represents a central element in $M$. By \S \ref{basic-min} (c), for any $l \in \BN$, $$(h m h \i)^{l \mathfrak i} \in (\CI_M W_K \CI_M) (\CI_M t^{l \mathfrak i v} \CI_M).$$ 

Let $N_1=\max_{K \subset \tilde \BS(M); W_K \text{ is finite }} \sharp W_K$. Let $i_{v, w}=(2 N_0+N_1+1) \mathfrak i$. Then for any $\a \in \Phi_{v, +}$, $\<i_{v, w} v, \a\> \ge 2 N_0+N_1+1$. Note that $m^{i_{v, w}}=h \i (g_1 g_2) h$ with $h \in \cup_{w' \in \tW(M); \ell(w') \le N_0} \CI_M \dot w' \CI_M$, $g_1 \in \cup_{u' \in \tW(M); \ell_M(u') \le N_1} \CI_M \dot u' \CI_M$ and $g_2 \in \CI_M t^{i_{v, w} v} \CI_M$. So 
\begin{align*} m^{i_{v, w}} N_{v, n} (m^{i_{v, w}}) \i &=h \i g_1 g_2 h N_{v, n} h \i g_2 \i g_1 \i h \\ & \subset h \i g_1 g_2 N_{v, n-N_0} g_2 \i g_1 \i h  \\ & \subset h \i g_1 N_{v, n-N_0+(2 N_0+N_1+1)} g_1 \i h \\ & \subset h \i N_{v, n-N_0+(2 N_0+N_1+1)-N_1} h \\ & \subset N_{v, v, n-N_0+(2 N_0+N_1+1)-N_1-N_0}=N_{v, n+1}.
\end{align*}

Similarly, $m^{-i_{v, w}} N^-_{v, n} m^{i_{v, w}} \subset N^-_{v, n+1}$. 

\section{The map $\bar i_\nu$}\label{3}

We define the induction map $\bar i_\nu$, which is the main object in this paper.

\begin{theorem}\label{levi}
Let $M$ be a semistandard Levi subgroup of $G$ and $\nu \in \aleph_M$ with $M=M_\nu$. Then

(1) For $m \in M$ and an open compact subgroup $\CK_M$ of $\CI_M$ with $m \CK_M \subset M(\nu)$, the map $$\d_{m \CK_M} \mapsto \d_{\nu}(m)^{-\frac{1}{2}} \frac{\mu_M(\CK_M)}{\mu_G(\CK_M \CK)} \d_{m \CK_M \CK}+[H, H]$$ from $H(M; \nu)$ to $\bar H(\bar \nu)$ is independent the choice of sufficiently small open compact subgroup $\CK$ of $G$

(2) The map $i_\nu: H(M; \nu) \to \bar H$ defined above induces a map $$\bar i_\nu: \bar H(M; \nu) \to \bar H.$$
\end{theorem}

\begin{remark}
Unlike the map $j_{v, n}$, the map $\bar i_\nu$ does not send $\bar H(M, M_n; \nu)$ to $\bar H(G, \CI_n; \bar \nu)$. One needs to replace $\CI_n$ by a smaller open compact subgroup of $G$. However, by the Iwahori-Matsumoto presentation of $\bar H(M, M_n; \nu)$ (\cite[Theorem 4.1]{hecke-1}) and Proposition \ref{+1}, there exists a positive integer $n'$ (depending on $\nu$) such that $\bar i_\nu: \bar H(M, M_n; \nu) \to \bar H(G, \CI_{n+n'}; \bar \nu)$ for any $n \in \BN$. 
\end{remark}

\begin{proof}
(1) Let $v$ be the $V$-factor of $\nu$. Let $w \in \tW(M)$ with $m \in \CI_M \dot w \CI_M$. Let $i_{v, w}$ be an positive integer in Proposition \ref{+1}. Let $l$ be a multiple of $i_{v, w} \ell(w)$ with $M_l \subset \CK_M$. By Proposition \ref{Levi-1} (2),  for any $n \in \BN$ and $g \in \CI_l$, we have $$\d_{m' g M_l \CI_{l+n}} \equiv \d_{m' M_l \CI_{l+n}} \mod [H, H].$$
Let $\CK, \CK'$ be open compact subgroups of $G$ with $\CK, \CK' \subset \CI_l$. Let $n \in \BN$ with $\CI_{l+n} \subset \CK, \CK'$. Now we have \begin{align*} \d_{m \CK_M \CK} &=\sum_{m' \in m \CK_M/M_l} \d_{m' M_l \CK} \equiv \sum_{m' \in m \CK_M/M_l} \frac{\mu_G(M_l \CK)}{\mu_G(M_l \CI_{l+n})} \d_{m' M_l \CI_{l+n}} \\ &=\frac{\mu_G(M_l \CK)}{\mu_G(M_l \CI_{l+n})} \d_{m \CK_M \CI_{l+n}} \mod [H, H].\end{align*}

As $\CK_M$ is stable under the right multiplication of $M_l$, we have $\mu_G(\CK_M \CI_{l+n})=\sharp(\CK_M/M_l) \mu_G(M_l \CI_{l+n})$ and $\frac{\mu_G(M_l \CK)}{\mu_G(M_l \CI_{l+n})}=\frac{\mu_G(\CK_M \CK)}{\mu_G(\CK_M \CI_{l+n})}$. Thus for any $n \in \BN$, we have $$\frac{\mu_M(\CK_M)}{\mu_G(\CK_M \CK)} \d_{m \CK_M \CK} \equiv \frac{\mu_M(\CK_M)}{\mu_G(\CK_M \CI_{l+n})} \d_{m \CK_M \CI_{l+n}} \mod [H, H].$$ Similarly, $\frac{\mu_M(\CK_M)}{\mu_G(\CK_M \CK')} \d_{m \CK_M \CK'} \equiv \frac{\mu_M(\CK_M)}{\mu_G(\CK_M \CI_{l+n})} \d_{m \CK_M \CI_{l+n}} \mod [H, H]$. Part (1) is proved.

(2) By \cite[\S 3.3 (2)]{hecke-1}, $[H(M), H(M)]=\oplus_{\nu \in \aleph_M} ([H(M), H(M)] \cap H(M)_{\nu})$, the kernel of the map $H(M)_\nu \to \bar H(M)_\nu$ is spanned by $\d_{m \CK_M}-{}^h \d_{m \CK_M}$ for $h, m \in M$ and open compact subgroup $\CK_M$ of $\CI_M$ such that $m \CK_m \subset M_\nu$. It remains to prove that $i_\nu(\d_{m \CK_M})=i_\nu({}^h \d_{m \CK_M})$. 

Set $m'=h m h \i$ and $\CK'_M=h \CK_M h \i$. By part (1), there exists a sufficiently small open compact subgroup $\CK$ of $G$ such that \begin{gather*} i_{\nu}(\d_{m \CK_M})\equiv \d_{\nu}(m)^{-\frac{1}{2}} \frac{\mu_M(\CK_M)}{\mu_G(\CK_M \CK)} \d_{m \CK_M \CK} \mod [H, H], \\ i_{\nu}(\d_{m' \CK'_M})\equiv \d_{\nu}(m')^{-\frac{1}{2}} \frac{\mu_M(\CK'_M)}{\mu_G(\CK'_M \CK')} \d_{m' \CK'_M \CK'} \mod [H, H].\end{gather*} Here $\CK'=h \CK h \i$. 

We have $\d_{m' \CK'_M \CK'}=\d_{h (m \CK_M \CK) h \i} \equiv \d_{m \CK_M \CK} \mod [H, H]$. Part (2) is proved. \qedhere







\end{proof}

\subsection{}\label{M-M'} In the rest of this section, we show that the maps $\bar i_*$ are compatible with conjugating the Levi subgroups. 

For any semistandard Levi subgroup $M$, we have a natural projection $$X_*(Z)_{\Gal(\bar F/F)}/\BZ \Phi_M^\vee \cong \Omega_M$$ and a natural map $V \mapsto V_+^M$. The natural action of $W_0$ on $X_*(Z)_{\Gal(\bar F/F)} \times V$ induces the following commutative diagram for any $w \in W_0$
\[
\xymatrix{
X_*(Z)_{\Gal(\bar F/F)} \times V \ar[r]^-{w \cdot} \ar[d] & X_*(Z)_{\Gal(\bar F/F)} \times V \ar[d] \\
\aleph_M \ar[r] & \aleph_{\dot w M \dot w \i}.
}
\]

We denote the induced map $\aleph_M \to \aleph_{\dot w M \dot w\i}$ still by $w \cdot$. If moreover, $w \in W^M$, i.e. $w$ sends the positive roots of $M$ to the positive roots of $\dot w M \dot w \i$, then we have $\dot w \CI_M \dot w \i=\CI_{\dot w M \dot w \i}$. By definition, the $M$-fundamental alcove is the unique $M$-alcove that contains the $G$-fundamental alcove. Since the conjugation by $\dot w$ sends the Iwahori-subgroup of $M$ to the Iwahori-subgroup of $\dot w M \dot w \i$, it also sends the $M$-fundamental alcove to the $\dot w M \dot w \i$-fundamental alcove, and thus induces a length-preserving map from $\tW(M)$ to $\tW(\dot w M \dot w \i)$. In particular, the conjugation by $w$ sends the minimal length elements of $\tW(M)$ (with respect to $\ell_M$) to the minimal length elements of $\tW(\dot w M \dot w \i)$ (with respect to $\ell_{\dot w M \dot w \i}$). Therefore, by the definition of Newton strata, we have that 

(a) Let $M$ be a semistandard Levi subgroup $M$ and $\nu \in \aleph_M$. Let $w \in W_0$ and $M'=\dot w M \dot w \i$, then $$\dot w M(\nu) \dot w \i=M'(w(\nu)).$$

\begin{proposition}\label{compatible}
Let $M$ be a semistandard Levi subgroup and $\nu \in \aleph_M$ and $w \in W_0$. Then for any $m \in M$, and an open compact subgroup $\CK_M$ of $\CI_M$ with $m \CK_M \subset M_\nu$ and $\dot w \CK_M \dot w \i \subset \CI_{\dot w M \dot w \i}$, we have
$$i_\nu(\d_{m \CK_M})=i_{w(\nu)}(\d_{\dot w m \CK_M \dot w \i}) \in \bar H.$$
\end{proposition}


\begin{proof}
The proof is similar to the proof of Theorem \ref{levi} (2). 

Set $M'=\dot w M \dot w \i$, $m'=\dot w m \dot w \i$ and $\CK_{M'}=\dot w \CK_M \dot w \i$. By Theorem \ref{levi} (1), there exists a sufficiently small open compact subgroup $\CK$ of $G$ such that \begin{gather*} i_{\nu}(\d_{m \CK_M})\equiv \d_{\nu}(m)^{-\frac{1}{2}} \frac{\mu_M(\CK_M)}{\mu_G(\CK_M \CK)} \d_{m \CK_M \CK} \mod [H, H], \\ i_{w(\nu)}(\d_{m' \CK_{M'}})\equiv \d_{w(\nu)}(m')^{-\frac{1}{2}} \frac{\mu_M(\CK_{M'})}{\mu_G(\CK_{M'} \CK')} \d_{m' \CK_{M'} \CK'} \mod [H, H].\end{gather*} Here $\CK'=\dot w \CK \dot w \i$. 

We have $\d_{m' \CK'_M \CK'}=\d_{\dot w (m \CK_M \CK) \dot w \i} \equiv \d_{m \CK_M \CK} \mod [H, H]$. The statement is proved. \qedhere

\end{proof}

\begin{corollary}
Let $M$ be a semistandard Levi subgroup of $G$ and $\nu \in \aleph_M$ with $M=M_\nu$. Then for any $w \in W_0$, $$\Im(\bar i_\nu: \bar H(M; \nu) \to \bar H)=\Im(\bar i_{w(\nu)}: \bar H(\dot w M \dot w \i; w(\nu)) \to \bar H).$$
\end{corollary}

\section{The image of the map $\bar i_\nu$}

The main result of this section is 

\begin{theorem}\label{BL}
Let $M$ be a semistandard Levi subgroup and $\nu \in \aleph_M$ with $M=M_\nu$. Then the image of the the map $\bar i_{\nu}: \bar H(M; \nu) \to \bar H$ equals $\bar H(\bar \nu)$. 
\end{theorem}

\smallskip

We first compare the Newton strata of $G$ and its Levi subgroups. 

\begin{proposition}\label{M-G-nu}
Let $M$ be a semistandard Levi subgroup and $\nu \in \aleph_M$ with $M_\nu=M$. Then we have $M(\nu) \subset G(\bar \nu)$. 
\end{proposition}




\begin{proof} The idea is similar to the proof of \cite[Theorem 2.1]{hecke-1}. 

By \S \ref{M-M'} (a), after conjugating by a suitable element in $W_0$, we may assume that $M$ is a standard Levi subgroup. Since $M=M_\nu$, the $V$-factor of $\nu$ is $G$-dominant. By the Newton decomposition of $G$ (\cite[Theorem 2.1]{hecke-1}), it suffices to prove that $M(\nu) \cap G(\nu')=\emptyset$ for any $\nu' \in \aleph$ with $\nu' \neq \bar \nu$. 

Let $\nu=(\t, v)$ and $\nu'=(\t', v')$. If the image of $\t$ in $\Omega$ does not equal to $\t'$, then $M(\nu) \cap G(\nu')=\emptyset$. Now we assume that the $\Omega$-factor matches. Since $\nu' \neq \bar \nu$, we have $v' \neq v$.  

By \cite[Remark 2.6]{hecke-1}, $$M(\nu)=\cup_{(x, K, u)} M \cdot \CI_M \dot u \dot x \CI_M, \quad G(\nu')=\cup_{(x', K', u')} G \cdot \CI \dot u' \dot x' \CI,$$ where $(x, K, u)$ runs over standard triples of $\tW(M)$ such that $u x \in \tW(M)_{\min}$ and $\pi_M(x)=\nu$,  $(x', K', u')$ runs over standard triples of $\tW$ such that $u' x' \in \tW_{\min}$ and $\pi(x')=\nu'$. 

If $M(\nu) \cap G(\nu') \neq \emptyset$, then there exists standard triples $(x, K, u)$ and $(x', K', u')$ as above and $h \in \CI_M \dot u \dot x \CI_M, h' \in \CI \dot u' \dot x' \CI, g \in G$ such that $g  h g \i=h'$. For any $n \in \BN$, we have $g h^n g \i=(h')^n$. By \S \ref{basic-min} (c), we have $$h^n \in (\CI_M W_K \CI_M) (\CI_M \dot x^n \CI_M), \quad (h')^n \in (\CI W_{K'} \CI) (\CI (\dot x')^n \CI).$$ Let $l >0$ with $l v, l v' \in X_*(Z)$. Suppose that $g \in \CI \dot z \CI$ for some $z \in \tW$. Then for any $n \in \BN$, we have $$\CI \dot z \CI t^{n l v} \CI (\CI W_K \CI) \CI \dot z \i \CI (\CI \dot W_{K'} \CI) \cap \CI t^{n l v'} \CI \neq \emptyset.$$

Similar to the argument in \cite[\S 2.6]{hecke-1}, this is impossible for $n \gg 0$. The statement is proved.
\end{proof} 

\begin{corollary}\label{in-G-nu}
The image of the map $\bar i_{\nu}$ is contained in $\bar H(\bar \nu)$. 
\end{corollary}

\begin{proof}
Let $m \in M$ and $\CK_M$ be an open compact subgroup of $\CI_M$ with $m \CK_M \subset M(\nu)$. By Proposition \ref{M-G-nu}, $m \CK_M \subset G(\bar \nu)$. Let $X$ be an open compact subset of $G$ with $m \CK_M \subset X$. By \cite[Theorem 3.2]{hecke-1}, there exists $n \in \BN$ such that $X \cap G(\bar \nu)$ is stable under the right multiplication by $\CI_n$. In particular, $m \CK_M \CI_n \subset G(\bar \nu)$. Thus $\bar i_{\nu}(\d_{m \CK_M}) \in \bar H(\bar \nu)$.  
\end{proof}

\subsection{} In order to prove the other direction, we use the notion of alcove elements in \cite{GHKR} and \cite{GHN}.

Let $w \in \tW$. We may regard $w \in \text{Aff}(V)$ as an affine transformation. Let $p: \text{Aff}(V)=V \rtimes GL(V) \to GL(V)$ be the natural projection map. Let $v \in V$. We say that $w$ is a $v$-alcove element if \begin{itemize}
\item $p(w)(v)=v$;

\item $N_v \cap \dot w \CI \dot w \i \subset N_v \cap \CI$. 
\end{itemize}

Note that the first condition implies that $\dot w M_v \dot w \i=M_v$. We have the following result. 

\begin{theorem}\label{Levi-2}
Let $w \in \tW$. If $w$ is a $\nu_w$-alcove element, then any element in $\CI \dot w \CI$ is conjugate by $\CI$ to an element in $\dot w \CI_{M_{\nu_w}}$. 
\end{theorem}

\begin{proof}
The basic idea is similar to the proof of \cite[Theorem 2.1.2]{GHKR}. 

Write $M$ for $M_{\nu_w}$ and $N$ for $N_{\nu_w}$. We start with the generic Moy-Prasad filtration $\CI=\CI[0] \supset \CI[1] \supset \cdots$. As explained in \cite[\S 6.2]{GHKR}, it is a filtration satisfying the following conditions:

(1) Each $\CI[r]$ is normal in $\CI$;

(2) For each $r$, either $\CI[r] \subset \CI_M \CI[r+1]$ or there exists a root $a \in \Phi-\Phi(M)$ and $s \in \RR$ such that $\CI[r]=X_{a+s} \CI[r+1]$ and $X_{a+s+\e} \subset \CI[r+1]$ for any $\e>0$. 

We show that each element $\dot w i_M i[r]$ with $i_M \in \CI_M$ and $i[r] \in I[r]$ is conjugate by an element in $\CI$ to an element in $\dot w \CI_M \CI[r+1]$ (and that the conjugator can be taken to be small when $r$ is large). 

If $\CI[r] \subset \CI_M \CI[r+1]$, then we may absorb the $\CI_M$ part into $i_M$. Otherwise, there exists a root $a$ outside $M$ such that $\CI[r]=X_{a+s} \CI[r+1]$ and $X_{a+s+\e} \subset \CI[r+1]$ for any $\e>0$. We prove the case where $a$ is a root in $N$. The case where $a$ is a root in $N^-$ can be proved in the same way. 

We have $i[r] \in u \CI[r+1]$ for some $u \in X_{a+s} \subset N_s$. Set $m=\dot w i_M$. By the definition of $P$-alcove elements, $m^{i} u (m^{i}) \i \subset N_s$ for all $i \in \BN$. As in the proof of Proposition \ref{Levi-1} (1),  $\dot w i_M i[r]$ is conjugate by elements in $N_s$ to elements in \begin{align*} m u \CI[r+1] &=(m u m \i) m \CI[r+1] \sim m \CI[r+1] (m u m \i)=m (m u m \i) \CI[r+1] \\ &=(m^{2} u (m^{2}) \i) m \CI[r+1] \sim \cdots \sim (m^{i} u (m^{i}) \i) m \CI[r+1] \sim\cdots. \end{align*} Here $\sim$ means conjugation by elements in $N_s$. 

By Proposition \ref{+1}, there exist $i \in \BN$ such that $m^{i} u (m^{i}) \i \subset \CI[r+1]$. Thus $\dot w i_M i[r]$ is conjugate by an element in $N_s$ to an element in $\dot w \CI_M \CI[r+1]$. 

Now we start with an element in $\dot w \CI$. The convergent product of the conjugators (for all $r$) is an element in $\CI$ and conjugates the given element to an element in $\dot w \CI_M$. 
\end{proof}

\subsection{Proof of Theorem \ref{BL}} By Corollary \ref{in-G-nu}, the image of $\bar i_\nu$ is contained in $\bar H(\bar \nu)$. Now we prove the other direction. By \cite[Corollary 4.2]{hecke-1}, $$\bar H=\sum_{w \in \tW_{\min}; \pi(w)=\bar \nu} \bar H_w,$$ where $H_w$ is the submodule of $H$ consisting of functions supported in $\CI \dot w \CI$ and $\bar H_w$ is the image of $H_w$ in $\bar H$. 

Let $w \in \tW_{\min}$ with $\pi(w)=\bar \nu$. By \cite[Lemma 4.4.3 and Proposition 4.4.6]{GHN}, $w$ is a $\nu_w$-alcove element. Set $M'=M_{\nu_w}$ and $\nu'=\pi_{M'}(w) \in \aleph_{M'}$. 

Let $i_{\nu', w}$ be a positive integer in Proposition \ref{+1}. By definition, $H_w$ is spanned by $\d_{g \CI_n}$ for $g \in \CI \dot w \CI$ and $n>i(\nu', w) \ell(w)$. By the proof of Theorem \ref{levi} (1), for any $n>i(\nu', w) \ell(w)$ and $g \in \dot w \CI_{M'}$, $\d_{g \CI_n}+[H, H]$ is contained in the image of $\bar i_{\nu'}$. 

Let $g \in \CI \dot w \CI$. By Theorem \ref{Levi-2}, there exists $i \in \CI$ and $g' \in \dot w \CI_{M'}$ such that $g=i g' i \i$. Then
$$\d_{g \CI_n}=\d_{i g' \CI_n i \i} \equiv \d_{g' \CI_n} \mod [H, H].$$ Therefore $\bar H_w$ is contained in the image of $\bar i_{\nu'}$. By Proposition \ref{compatible}, $\bar H_w$ is also contained in the image of $\bar i_{\nu}$. 

\section{Adjunction with the Jacquet functor}

\subsection{} Let $R$ be an algebraically closed field of characteristic $\neq p$. Set $H_R=H \otimes_{\BZ[\frac{1}{p}]} R$, $\bar H_{R}=\bar H \otimes_{\BZ[\frac{1}{p}]} R$ and $\bar H_R(\nu)=\bar H(\nu) \otimes_{\BZ[\frac{1}{p}]} R$. Recall that $\fkR(G)_R$ is the $R$-vector space with basis the isomorphism classes of irreducible smooth admissible representations of $G$ over $R$. We consider the trace map $$\Tr^G_R: \bar H_R \to \fkR(G)_R^*.$$

Similarly, for any semistandard Levi subgroup $M$, we have $$\Tr^M_R: \bar H_R(M) \to \fkR(M)_R^*.$$ 

Let $v \in V$ and $M=M_v$. Let $r_{v, R}: \fkR(G)_R \to \fkR(M)_R$ be the (normalized) Jacquet functor. Note that the Jacquet functor does not only depend on the Levi $M$, but also depends on the direction $v$ (or equivalently, the parabolic subgroup $P_v$ with Levi factor $M$). The following result is proved by Bushnell in \cite[Corollary 1]{Bu}. 

\begin{proposition}\label{Bu-adj}
Let $n \in \BN$. Let $v \in V$ and $M=M_v$. Then for any $f \in H^v_R(M, M_n)$, and $\pi \in \fkR_{\CI_n}(G)_R$, we have $$\Tr^M_R(f, r_{v, R}(\pi))=\Tr^G_R(j_{v, n}(f), \pi).$$
\end{proposition}

\

The main result of this section is the following adjunction formula. 

\begin{theorem}\label{adj}
Let $M$ be a semistandard Levi subgroup and $\nu \in \aleph_M$. Suppose that $M=M_\nu$. Then for any $f \in \bar H_R(M; \nu)$ and $\pi \in \fkR(G)_R$, we have $$\Tr^M_R(f, r_{\nu, R}(\pi))=\Tr^G_R(\bar i_{\nu}(f), \pi).$$
\end{theorem}

\subsection{}\label{a-b}
Let $(x, K, u)$ be a standard triple of $\tW(M)$ such that the Newton point of $x$ is $v$. Let $\mathfrak i$ be the smallest positive integer with $\mathfrak i v \in X_*(Z)$. Let $i \in \BN$ such that for any $\a \in \Phi_{v, +}$, $\<i \mathfrak i v, \a\> \ge \sharp W_K+(\mathfrak i-1) \ell(x)+1$. Let $l \ge i \mathfrak i$. Then $l=i' \mathfrak i+j$ for some $i' \ge i$ and $0 \le j<\mathfrak i$. Then for any $m_1, \cdots, m_l \in \CI_M \dot u \dot x \CI_M$, by \S \ref{basic-min} (c), we have $$m_1 m_2 \cdots m_l \in (\CI_M W_K \CI_M) (\CI_M \dot x^j \CI_M) (\CI_M t^{i' \mathfrak i v} \CI_M).$$

Note that for $g \in \CI t^{i' \mathfrak i v} \CI$, $g N_n g \i \subset N_{n+\sharp W_K+(\mathfrak i-1) \ell(x)+1}$. Also $(\CI W_K \CI) (\CI \dot x^j \CI)  \subset \cup_{w \in \tW; \ell(w) \le \sharp W_K+(\mathfrak i-1) \ell(x)} \CI \dot w \CI$. Thus $(m_1 \cdots m_l) N_n (m_1 \cdots m_l) \i \subset N_{n+1}.$ Similarly $(m_1 \cdots m_l) \i N^-_n (m_1 \cdots m_l)  \subset N^-_{n+1}.$ Therefore, 

(a) Let $l \ge i \mathfrak i$ and $m_1, \cdots, m_l \in \CI_M \dot u \dot x \CI_M$, then $m_1 m_2 \cdots m_l$ is a $P_v$ strictly positive element. 

Moreover, for any $n, l' \in \BN$ and $m_1, \cdots, m_{l'} \in \CI_M \dot u \dot x \CI_M$, we have \begin{gather*} (m_1 \cdots m_{l'}) N_{n+\sharp W_K+(\mathfrak i-1) \ell(x)} (m_1 \cdots m_{l'}) \i \subset N_{n},  \\ (m_1 \cdots m_{l'}) \i N^-_{n+\sharp W_K+(\mathfrak i-1) \ell(x)} (m_1 \cdots m_{l'}) \subset N^-_{n}.\end{gather*}

One deduces that 

(b) Let $n, l' \in \BN$, and $g_1, \cdots g_{l'} \in N_{n+\sharp W_K+(\mathfrak i-1) \ell(x)} \CI_M \dot u \dot x \CI_M N^-_{n+\sharp W_K+(\mathfrak i-1) \ell(x)}$. Then $g_1 \cdots g_{l'} \in N_n M N^-_n$. 

\subsection{Proof of Theorem \ref{adj}}
By \cite[Theorem 4.1 \& \S 4.6]{hecke-1}, it suffices to prove it for locally constant functions on $M$, supported  in $M \dot u \dot x M$, where $(x, K, u)$ is a standard triple of $\tW(M)$ and the Newton point of $x$ is $v$. 

Let $n>\sharp W_K+(\mathfrak i-1) \ell(x)$ such that $\pi \in \fkR_{\CI_n}(G)_R$. It is enough to consider the function $f=\d_{M_n m M_n}$, where $m \in M \dot u \dot x M$. 

Let $n' \gg n$ and $\tilde f=\frac{\d_v(m)^{-\frac{1}{2}}}{\mu_N(N_{n'}) \mu_{N^-}(N^-_{n'})} \d_{N_{n'} M_n m M_n N^-_{n'}}$. By Theorem \ref{levi} (1), $\tilde f$ represents the element $\bar i_v(f) \in \bar H$.  By Casselman's trick \cite[Corollary 4.2]{Ca}, it suffices to prove that for $l \gg 0$, $\Tr^M_R(f^l, r_v(\pi))=\Tr^G_R(\tilde f^l, \pi)$. 

Let $p_M: (M_n m M_n)^l \to M$ and $p_G: (N_{n'} M_n m M_n N^-_{n'})^l \to G$ be the multiplication map. Since $l \gg 0$, by \S\ref{a-b} (a) and (b), any element in $\Im(p_M)$ is $P_\nu$ strictly positive and $$\Im(p_G) \subset N_n \Im(p_M) N^-_n \cong N \times \Im(p_M) \times N^-.$$ We have the following commutative diagram 
\[
\xymatrix{(N_{n'} M_n m M_n N^-_{n'})^l \ar[r]^-{p_G} \ar[d]_-{pr^l} & \Im(p_G) \ar[d]^-{pr_1} \\ (M_n m M_n)^l \ar[r]^-{p_M} & \Im(p_M),}
\]
where $pr: N \times M \times N^- \to M$ is the projection map and $pr_1$ is the restriction of $pr$ to $\Im(p_G)$. 

Let $m' \in \Im(p_M)$. Then \begin{align*} \mu_{G^l}(p_G \i pr_1 \i (M_n m' M_n)) &=\mu_{G^l}((pr^l) \i p_M \i (M_n m' M_n)) \\ &=\mu_N(N_{n'})^l \mu_{N^-}(N^-_{n'})^l \mu_{M^l}(p_M \i (M_n m' M_n)). \end{align*} By Proposition \ref{Levi-1} (2), $\d_{i \CI_{n'} M_n m' M_n \CI_{n'} i'} \equiv \d_{\CI_{n'} M_n m' M_n \CI_{n'}} \mod [H, H]$ for any $i \in N_n$ and $i' \in N^-_n$.
Thus \begin{align*} \tilde f^l  & \equiv \frac{\d_v(m)^{-\frac{l}{2}}}{\mu_N(N_{n'})^l \mu_{N^-}(N^-_{n'})^l} \sum_{m' \in M_n\backslash M/M_n} \frac{\mu_{G^l}(p_G \i pr_1 \i (M_n m' M_n))}{\mu_G(pr_1 \i(M_n m M_n))} \d_{pr_1 \i(M_n m M_n)} \\ & \equiv \sum_{m' \in M_n\backslash M/M_n} \d_v(m)^{-\frac{l}{2}} \frac{\mu_{M^l}(p_M \i (M_n m' M_n))}{\mu_G(pr_1 \i(M_n m M_n))} \d_{pr_1 \i(M_n m M_n)}\\ & \equiv \sum_{m' \in M_n\backslash M/M_n} \d_v(m)^{-\frac{l}{2}} \frac{\mu_{M^l}(p_M \i (M_n m' M_n))}{\mu_G(\CI_{n'} M_n m' M_n \CI_{n'})} \d_{\CI_{n'} M_n m' M_n \CI_{n'}}   \mod [H, H].\end{align*} 

On the other hand, $$f^l=\sum_{m' \in M_n \backslash M/M_n} \frac{\mu_{M^l}(p_M \i (M_n m' M_n))}{\mu_M(M_n m' M_n)} \d_{M_n m' M_n}.$$ By Corollary \ref{str+}, we have \begin{align*} j_{v, n}(f^l) & \equiv j_{v, n'}(f^l) \\ &=\sum_{m' \in M_n \backslash M/M_n} \d_v(m)^{-\frac{l}{2}} \frac{\mu_{M^l}(p_M \i (M_n m' M_n))}{\mu_M(M_n m' M_n)} \frac{\mu_M(M_{n'})}{\mu_G(\CI_{n'})} \d_{\CI_{n'} M_n m' M_n \CI_{n'}} \mod [H, H].\end{align*} 

Since the elements in $M_n m' M_n$ are $P_v$ strictly positive, we have $\CI_{n'} M_n m' M_n \CI_{n'}=N_{n'} (M_n m' M_n) N^-_{n'}$ and $$\mu_G(\CI_{n'} M_n m' M_n \CI_{n'})=\mu_{N} (N_{n'}) \mu_{N^-}(N^-_{n'}) \mu_M(M_n m' M_n)=\frac{\mu_G(\CI_{n'})}{\mu_M(M_{n'})} \mu_M(M_n m' M_n).$$ So $\tilde f^l \equiv j_{v, n}(f^l) \mod [H, H]$ and $\Tr^M_R(f^l, r_v(\pi))=\Tr^G_R(\tilde f^l, \pi)$. 

\section{The kernel of the trace map} 

\subsection{} Let $M$ be a semistandard Levi subgroup of $G$. Let $M^0$ be the subgroup of $G$ generated by the parahoric subgroups of $M$. Then we have $M/M^0 \cong \Omega_M$. Let $\Psi(M)_R=\Hom_{\BZ} (M/M^0, R^\times)$ be the torus of unramified characters of $M$. 

Let $i_{M, R}: \fkR(M)_R \to \fkR(G)_R$ be the induction functor. Then for any $\s \in \fkR(M)_R$ and $f \in \bar H_R$, the map $$\Psi(M)_R \to R, \qquad \chi \mapsto \Tr_R(f, i_{M, R}(\s \circ \chi))$$ is an algebraic function over $\Psi(M)_R$. 

\subsection{} Let $v \in V$ and $M=M_v$. Recall that \begin{gather*}\tag{a} \bar H(M; v)=\oplus_{\nu_M \in \aleph_M; \nu=(\t_M, v) \text{ for some } \t_M \in \Omega_M} \bar H(M; \nu), \\ \tag{b} \bar H(\bar v)=\oplus_{\nu \in \aleph; \nu=(\t, \bar v) \text{ for some } \t \in \Omega} \bar H(\nu). \end{gather*}

Note that if $\t_M, \t'_M \in \Omega_M$ are mapped under $\k$ to the same element in $\Omega$, then they differ by a central cocharacter of $M$. By the definition of the map $\pi=(\k, \bar \nu)$, if both $(\t_M, v)$ and $(\t'_M, v)$ are in the image of $\pi_M$ and that $\k(\t_M)=\k(\t'_M)$, then $\t_M=\t'_M$. In other words, there is a natural bijection between the components appear on the right hand sides of (a) and (b). We define $$\bar i_v=\oplus_{\nu_M \in \aleph_M; \nu=(\t_M, v) \text{ for some } \t \in \Omega_M} \bar i_{\nu}: \bar H(M; v) \mapsto \bar H(\bar v).$$


\begin{theorem}\label{newton-ker}
Let $v \in V$ and $M=M_v$. Let $f \in \bar H(\bar v)$. If $\Tr^G_R(f, i_{M, R}(\s))=0$ for all $\s \in \fkR(M)_R$, then $f \in \bar i_v (\ker \Tr^M_R)$. 
\end{theorem}

\begin{proof}
For $\s \in \fkR(M)_R$ and $\chi \in \Psi(M)_R$, the map $$\chi \mapsto \Tr^G(\bar i_{v}(f), i_{M, R}(\s \circ \chi))$$ is an algebraic function on $\chi$. We consider its ``positive part'', i.e. the linear combination of the terms $\<\chi, \l\>$ for dominant coweight $\l$. It is obvious that if an algebraic function is zero, then its ``positive part'' is also zero. 

By the Mackey formula \cite[\S 5.5]{Vi96}, we have 
\begin{align*} \Tr^G_R(\bar i_{v}(f), i_{M, R}(\s \circ \chi)) &=\Tr^M_R(f, r_{M, R} \circ i_{M, R}(\s \circ \chi)) \\ &=\sum_{w \in {}^M W^M} \Tr^M_R(f, i^M_{M \cap {}^w M, R} \circ \dot w \circ r^M_{M \cap {}^{w \i} M, R}(\s \circ \chi) \\ &=\sum_{w \in {}^M W^M} \Tr^M_R(f, i^M_{M \cap {}^w M, R} (\dot w \circ r^M_{M \cap {}^{w \i} M, R}(\s) \circ {}^{\dot w} \chi) ).
\end{align*}
As $w \in {}^M W^M$ and $M=M_v$, $w(v)$ is dominant if and only if $w=1$. Therefore the ``positive part'' of $\Tr^G(\bar i_{v}(f), i_{M, R}(\s \circ \chi))$ is $\Tr^M_R(f, \s \circ \chi)$. 

Therefore if $\Tr^G_R(f, i_{M, R}(\s))=0$ for any $\s \in \fkR(M)_R$ and $\chi \in \Psi(M)_R$, then $\Tr^M_R(f, \s \circ \chi)=0$ for any $\s \in \fkR(M)_R$ and $\chi \in \Psi(M)_R$. Hence $f \in \ker \Tr^M_R$. \end{proof}

\begin{corollary}
Let $v \in V$ and $M=M_v$. Then $$\bar i_{v} \i (\ker \Tr^G_R \mid_{\bar H_R(\bar v)})=\ker \Tr^M_R \mid_{\bar H_R(M; v)}.$$
\end{corollary}

\begin{proof}
If $f \in \ker \Tr^M_R$, then $\Tr^M_R(f, r_{M, R} (\pi))=0$ for any $\pi \in \fkR(G)_R$. By Theorem \ref{adj}, $\Tr^G_R(\bar i_{v}(f), \pi)=0$. Thus $\bar i_{v}(f) \in \ker \Tr^G_R$. The other direction follows from Theorem \ref{newton-ker}.
\end{proof}

\begin{theorem}\label{newton-ker1}
We have $\ker \Tr^G_R=\bigoplus_{v \in V_+} (\ker \Tr^G_R \cap \bar H_R(v)).$
\end{theorem}

\begin{remark}
In general, $\oplus_{\nu \in \aleph} (\ker \Tr^G_R \cap \bar H_R(\nu)) \subset \ker \Tr^G_R$. However, the equality may not hold. For example, if $\Omega=\{1, \t\}$ is finite of order $2$ and characteristic of $R$ is also $2$, then for any $\l \in X_*(Z)_+$ and $f \in \bar H(\l)$, we have $f+\t f \in \ker \Tr^G_R$.
\end{remark}

\begin{proof}
The idea is similar to the proof of \cite[Theorem 7.1]{CH}. 

Let $f=\sum_{v \in V_+} a_v f_v \in \ker \Tr^G_R$, where $f_v \in \bar H_v$ and $a_v \in R$. Let $M$ be a minimal standard Levi subgroup such that $a_v \neq 0$ for some $v \in V_+$ with $M=M_v$. Then for $\s \in R(M)$ and $\chi \in \Psi(M)_R$, we have \[\tag{a} \Tr^G_R(f, i_{M, R}(\s \circ \chi))=\sum_{v \in V_+; M=M_v} a_v \Tr^G_R(f_v, i_{M, R}(\s \circ \chi))+\sum_{v \in V_+; M \neq M_v} a_v \Tr^G_R(f_v, i_{M, R}(\s \circ \chi)).\]

This is an algebraic function on $\Psi(M)_R$. Note that in (a), the first part is more regular in $\Psi(M)_R$ than the second part. Therefore we have $$\sum_{v \in V_+; M=M_v} a_v \Tr^G_R(f_v, i_{M, R}(\s \circ \chi))=0$$ for all $\s \in R(M)$ and $\chi \in \Psi(M)_R$. As an algebraic function on $\Psi(M)_R$, the ``leading term'' of $\Tr^G_R(f_v, i_{M, R}(\s \circ \chi))$ is a multiple of $\<v, \chi\>$. Hence $a_v \Tr^G_R(f_v, i_{M, R}(\s \circ \chi))=0$ for every $v \in V_+$ with $M=M_v$. By Theorem \ref{newton-ker}, $$a_v f_v \in \bar i_v(\ker \Tr^M_R \mid_{\bar H(M; v)}).\qedhere$$ 
\end{proof}

Finally, we have

\begin{theorem}\label{inj-i}
Assume that $\text{char}(F)=0$. Let $M$ be a semistandard Levi subgroup and $\nu \in \aleph_M$ with $M=M_\nu$. Then the map $$\bar i_\nu: \bar H(M; \nu) \xrightarrow{\cong} \bar H(\bar \nu)$$ is an isomorphism.
\end{theorem}

\begin{proof}
Let $f \in \ker \bar i_\nu$. Set $\tilde f=f \otimes 1 \in \bar H_{\BC}(M; \nu)$. By Theorem \ref{newton-ker} (2), we have $\tilde f \in \ker \Tr^M_\BC$. By the spectral density theorem \cite[Theorem 0]{Kaz}, $\tilde f=0 \in \bar H(M)_{\BC}$. By \cite{Vi2}, $\bar H(M)$ is free. Hence $f=0 \in \bar H(M)$.
\end{proof}


\begin{thebibliography}{99}

\bibitem{BT}
F. Bruhat and J. Tits, \emph{Groupes r\'eductifs sur un corps local I}, Inst. Hautes \'Etudes Sci. Publ. Math. 41 (1972), 5--251.


\bibitem{Bu}
C. J. Bushnell, \emph{Representations of reductive $p$-adic groups: localization of Hecke algebras and applications}, J. London Math. Soc. (2) 63 (2001), 364--386.

\bibitem{BK}
C. J. Bushnell and P.C. Kutzko, \emph{Smooth representations of reductive $p$-adic groups: structure theory via types}, Proc. London Math. Soc. (3) 77 (1998), 582--634.

\bibitem{Ca}
W. Casselman, \emph{Characters and Jacquet modules} Math. Ann. \textbf{230} (1977), 101--105.

\bibitem{CH}
D. Ciubotaru and X. He, \emph{Cocenters and representations of affine Hecke algebra}, arXiv:1409.0902, to appear in JEMS. 

\bibitem{hecke-3}
D. Ciubotaru and X. He, \emph{Cocenters of $p$-adic groups, III: Elliptic cocenter and rigid cocenter}, in preparation.

\bibitem{Dat}
J.-F.~Dat,
\emph{On the $K_0$ of a $p$-adic group}, Invent. Math. {\bf 140} (2000), no. 1, 171--226. 

\bibitem{GHKR} 
U.~G\"{o}rtz, T.~Haines, R.~Kottwitz, D.~Reuman, {\em
Affine Deligne-Lusztig varieties in affine flag varieties}, Compos. Math.~\textbf{146} (2010), 1339--1382.

\bibitem{GHN}
U. G\"ortz, X. He and S. Nie, \emph{$P$-alcoves and nonemptiness of affine Deligne-Lusztig varieties}, Ann. Sci. \`Ecole Norm. Sup. 48 (2015), 647--665.

\bibitem{L}
G. Lusztig,  \emph{Affine Hecke algebras and their graded version}, J. Amer. Math. Soc. 2 (1989), no. 3, 599--635.

\bibitem{hecke-1}
X. He, \emph{Cocenters of $p$-adic groups, I: Newton decomposition}, 1610.04791.

\bibitem{HN1}
X. He and S. Nie, \emph{Minimal length elements of extended affine Weyl group}, Compos. Math. 150 (2014), 1903--1927.

\bibitem{HN2}
X. He and S. Nie, \emph{$P$-alcoves, parabolic subalgebras and cocenters of affine Hecke algebras}, Selecta. Math. 21 (2015), 995--1019.

\bibitem{Kaz}
D.~Kazhdan, \emph{Cuspidal geometry of $p$-adic groups}, J. Analyse Math. {\bf 47} (1986), 1--36.

\bibitem{MP}
A. Moy, G. Prasad, \emph{Unrefined minimal $K$-types for $p$-adic groups} Invent. Math. 116 (1994), 393--408. 

\bibitem{Vi96}
M.-F. Vign{\'e}ras, \emph{Repr\'esentations {$l$}-modulaires d'un groupe r\'eductif{$p$}-adique avec {$l\ne p$}}, Progress in Math. 137, Birkh\"auser, 1996.

\bibitem{Vi2}
M.-F. Vign{\'e}ras, \emph{Int\'egrales orbitales modulo {$l$} pour un groupe r\'eductif {$p$}-adique}, Algebraic geometry: Hirzebruch 70 (Warsaw, 1998), Amer. Math. Soc., Providence, RI, 1999, 241, 327--338.

\bibitem{V}
M.-F. Vign{\'e}ras, \emph{The pro-$p$-Iwahori-Hecke algebra of a reductive $p$-adic group, II}, Compos. Math. 152 (2016), no. 4, 693--753.

\end{thebibliography}
\end{document}